\pdfoutput=1 % for arXiv to use pdflatex instead of latex+dvipdf

\documentclass[a4paper,twoside,reqno]{article}

\usepackage{a4wide}
\usepackage{amssymb}
\usepackage{amsmath}
\usepackage{mathtools}
\usepackage{amsthm}

\usepackage{mathrsfs}

\usepackage{enumitem}
\setlist[enumerate,1]{label=(\alph*)}
\setlist[enumerate,2]{label=(\roman*)}

\usepackage{url}
\usepackage[hidelinks]{hyperref}

\usepackage[usenames,dvipsnames]{xcolor}
\usepackage{mathtools}
\mathtoolsset{showonlyrefs}

\usepackage{dsfont}

\theoremstyle{definition}

\theoremstyle{remark}

\swapnumbers
\theoremstyle{plain}

\newtheorem{theorem}{Theorem}[section]
\newtheorem*{theorem*}{Theorem}

\newtheorem{lemma}[theorem]{Lemma}

\theoremstyle{remark}
\newtheorem{remark}[theorem]{Remark}
\newtheorem*{remark*}{Remark}

\theoremstyle{definition}
\newtheorem{definition}[theorem]{Definition}

\usepackage{authoraftertitle}

\usepackage[only,llbracket,rrbracket]{stmaryrd}

\newcommand{\End}[1]{ \mathrm{End}({#1}) }

\newcommand{\ccspace}[1]{\mathscr{K}(#1)}

\newcommand{\ograss}[2]{\mathbf{G}_0(#1,#2)}

\newcommand{\PC}[1]{\mathbf{P}_{#1}}

\newcommand{\IC}{\mathbf{I}}

\newcommand{\Mass}{\mathbf{M}}

\ifdefined\textint
    \renewcommand{\textint}[2]{{\textstyle\int_{#1}^{#2}}}
\else    
    \newcommand{\textint}[2]{{\textstyle\int_{#1}^{#2}}}
\fi

\ifdefined\textsum
    \renewcommand{\textsum}[2]{{\textstyle\sum_{#1}^{#2}}}
\else  
    \newcommand{\textsum}[2]{{\textstyle\sum_{#1}^{#2}}}
\fi

\newcommand{\natp}{\mathbb{N}}
\newcommand{\nat}{\natp \cup \{0\}}

\newcommand{\UPC}{\mathrm{UPC}}
\newcommand{\AUE}{\mathrm{AUE}}
\newcommand{\UQC}{\mathrm{UQC}}

\newcommand{\integers}{\mathbf{Z}}
\newcommand{\Z}{\integers}

\newcommand{\R}{\mathbf{R}}

\newcommand{\Q}{\mathbf{Q}}

\newcommand{\LM}{\mathscr{L}}

\newcommand{\HM}{\mathscr{H}}

\newcommand{\restrict}{ \mathop{ \rule[1pt]{.5pt}{6pt} \rule[1pt]{4pt}{0.5pt} }\nolimits }

\newcommand{\ud}{\ensuremath{\,\mathrm{d}}}

\newcommand{\uD}{\ensuremath{\mathrm{D}}}

\newcommand{\id}[1]{\mathds{1}_{#1}}

\newcommand{\tbwedge}{{\textstyle \boldsymbol{\bigwedge}}}

\newcommand{\Dirac}[1]{\boldsymbol{\delta}_{#1}}

\newcommand{\Lbrack}{\llbracket}
\newcommand{\Rbrack}{\rrbracket}

\DeclareMathOperator{\Hom}{Hom}

\DeclareMathOperator{\asssp}{space}

\newcommand{\cnt}[1]{\mathscr{C}^{#1}}

\newcommand{\orthproj}[2]{\mathbf{O}^\ast({#1},{#2})}

\DeclareMathOperator{\spt}{spt}

\DeclareMathOperator{\conv}{conv}

\DeclareMathOperator{\ray}{ray}

\DeclareMathOperator{\cone}{cone}

\DeclareMathOperator{\Lip}{Lip}

\newcommand{\without}{\!\sim\!}
\newcommand{\qspace}[1]{\mathcal{A}_{{#1}}}

\DeclareMathOperator{\im}{im}

\date{\today}

\title{
    Equivalence of Uniform Polyconvexity and Almgren Uniform Ellipticity for 
    Lipschitz $Q$-Graph Test Pairs
}

\author{
    Maciej Lesniak
}

\hypersetup{
  unicode=true,
  pdfauthor={\MyAuthor},
  pdftitle={\MyTitle},
  pdfsubject={},
  pdfkeywords={},
  pdfproducer={},
  pdfcreator={}
  pdfinfo={
    orcid_ML={0009-0002-4657-4213},
  }
}

\begin{document}

\maketitle

\begin{abstract}
    We investigate the relationship between uniform polyconvexity of anisotropic geometric 
    integrands and Almgren's uniform ellipticity. We first establish the converse 
    implication for uniform ellipticity with respect to polyhedral test pairs, 
    thereby strengthening earlier results. Our main theorem shows that uniform 
    polyconvexity is equivalent to Almgren's uniform ellipticity with respect to 
    Lipschitz $Q$-graph test pairs, building on techniques developed by De Rosa, Lei, 
    and Young. As a consequence, we show that for a classical integrand, uniform 
    polyconvexity is equivalent to uniform quasiconvexity of the associated 
    $Q$-integrand for every~$Q \in \natp$.
\end{abstract}

\section{Introduction}
\hfill

In geometric measure theory one is led to the study of minimisers and critical
points of functionals defined on $k$--dimensional currents or varifolds
in~$\R^n$. In particular, if $\ograss{n}{k}$ denotes the Grassmannian of
oriented $k$--planes in $\R^n$, a continuous \emph{geometric integrand}
$\Psi : \ograss{n}{k} \to (0,\infty)$ determines an anisotropic energy
functional
\begin{displaymath}
    E_{\Psi}(T) = \textint{}{} \Psi(\vec T(x)) \ud \|T\|(x) \,,
\end{displaymath}
where $\|T\|$ denotes the variation measure of the rectifiable current $T$ and
$\vec T$ is the associated orientation map; see~\cite[2.5.12, 4.1.5,
4.1.7]{Federer1969}. If $\Psi$ is even, $E_{\Psi}$ extends naturally to
$k$--dimensional varifolds; see~\cite{Allard1972}. Based on Almgren's
notion~\cite[1.2, 1.6(2)]{Almgren1968}, given a family $\mathcal{P}$ of pairs of
$k$-dimensional surfaces in~$\R^n$ we call $\Psi$ \emph{elliptic with respect to
  $\mathcal{P}$} if
\begin{displaymath}
    E_{\Psi}(S) > E_{\Psi}(D)
    \quad \text{for all $(S,D) \in \mathcal{P}$} \,.
\end{displaymath}
At this stage $\mathcal{P}$ may be arbitrary but one should think that $D$ is
always a flat $k$-dimensional disc (or a cube) and $S \ne D$ is a
$k$-dimensional surface attached to the relative boundary of~$D$. In~case there
exists $c > 0$ such that
\begin{displaymath}
    E_{\Psi}(S) - E_{\Psi}(D) \ge c \bigl( \HM^k(S) - \HM^k(D) \bigr)
    \quad \text{for all $(S,D) \in \mathcal{P}$} 
\end{displaymath}
we call $\Psi$ \emph{uniformly elliptic with respect to $\mathcal{P}$}. Uniform
ellipticity is the condition that allowed Almgren to prove full partial
regularity of minimisers of~$E_{\Psi}$ in classes of rectifiable sets and
rectifiable currents with coefficients in an abelian group; see~\cite[1.4,
1.7]{Almgren1968}.

Parallel to this geometric setting, in the classical calculus of variations one
studies minimisers and critical points of functionals of the type
\begin{displaymath}
    E_{\psi}(u) = \textint{\Omega}{} \psi(\uD u(x)) \ud \LM^k(x) \,,
\end{displaymath}
where $\Omega \subseteq \R^k$ is a bounded open set, $u : \Omega \to \R^{n-k}$
is a weakly differentiable map, and $\psi : \Hom(\R^k,\R^{n-k}) \to (0,\infty)$
is a given \emph{classical integrand}; cf.~\ref{def:classical_integrand}.
In~this context one introduces the notions of quasi-convexity, polyconvexity,
and ellipticity of~$\psi$ % in order to be able to prove existence and regularity
% of solutions to variational problems involving~$E_{\psi}$
and these three notions differ; see e.g.~\cite[Chapter~5]{Dacorogna2008}.

The geometric ellipticity condition depends strongly on the choice of~$\mathcal{P}$,
i.e., on the class of~\emph{admissible competitors}. This phenomenon is
described in more detail in the introduction to our previous paper;
see~\cite{Lesniak_2025} and references therein. Here we focus on two particular
choices for~$\mathcal{P}$: \emph{polyhedral test pairs} constructed from
polyhedral currents with real coefficients and \emph{graph test pairs} being
pairs of graphs of certain $Q$-valued functions;
cf.~\ref{def:polyhedral_test_pairs} and~\ref{def:graph_test_pair}.

When competitors in~$\mathcal{P}$ are rectifiable currents, uniformly convex
norms on~$\tbwedge_k\R^n$ induce elliptic integrands (we identify $\ograss nk$
with the set of unit simple $k$-vectors in $\R^n$);
cf.~\cite[5.1.2]{Federer1969}. In this setting Burago and
Ivanov~\cite{Burago2004} proved that semi--ellipticity is equivalent to
polyconvexity. De Rosa, Lei, and Young reproved that theorem
and showed also that polyconvexity of a classical integrand~$\psi$ and
quasiconvexity of the associated Q–integrand $\bar{\psi}_Q$ for every
$Q \in \natp$ are equivalent; cf.~\cite{Rosa2023}.

In our previous paper~\cite[Theorem~5.10]{Lesniak_2025}, it was shown that
Almgren uniform ellipticity with respect to polyhedral $k$-chains implies
uniform polyconvexity of the associated geometric integrand. The first result of
the present work establishes the converse implication, yielding an equivalence
between the two conditions.

\begin{theorem*}[see \ref{thm:main1}]
    Let $\Psi : \ograss{n}{k} \to \R$ be a Lipschitz geometric integrand and let
    $c>0$. Then $\Psi$ is Almgren uniformly elliptic with constant~$c$ with
    respect to polyhedral $k$-chains with real coefficients if and only if
    $\Psi$ is uniformly polyconvex with the same constant $c$.
\end{theorem*}

Our second result establishes an analogous relation for the uniform counterparts
with respect to $Q$--graph test pairs.

\begin{theorem*}[see \ref{corollary:thm2}]
    Let $\Psi : \ograss{n}{k}^{+} \to \R$ be a Lipschitz geometric integrand and
    let $c>0$. Then $\Psi$ is Almgren uniformly elliptic with respect to
    $Q$--graph test pairs if and only if $\Psi$ is uniformly polyconvex (both
    with the same constant $c$).
\end{theorem*}
The preceding results suggest that uniform polyconvexity is a natural condition 
underlying ellipticity in both the classical and multigraph settings.
This perspective motivates the study of analogous conditions in the
multiple--valued framework introduced by Almgren for the study of
area--minimizing currents; see~\cite{Almgren2000e}. This leads to considering
$Q$-valued maps $v : T \to \qspace{Q}(T^{\perp})$ and related $Q$-energies,
where $Q$ is a~fixed integer and $T$ is a~$k$-dimensional subspace of~$\R^n$;
cf.~\cite{Almgren2000e, De_Lellis_2011_Q_valued_functions_revisited,
  delellis2013multiplevaluedfunctionsintegral}. Given a geometric integrand
$\Psi$ and $Q \in \natp$ one can define the associated classical integrand
$\psi$ and also a~$Q$-integrand $\bar{\psi}_Q$ so that
$E_{\Psi}(S) = E_{\psi}(u)$ if $S$ is the graph of $u$ and
$E_{\Psi}(R) = E_{\bar{\psi}_Q}(v)$ if $R$ is the $Q$-graph of $v$;
cf.~\S{\ref{section:q_energies}}.

In~\cite{De_Lellis_2011_Q_valued_functions_revisited}, De Lellis and Spadaro
revisited Almgren's theory and developed a self--contained analytic framework
for Dir--minimizing $Q$--valued functions. Further developments were obtained in
\cite{delellis2013multiplevaluedfunctionsintegral}, where it was shown that a
Lipschitz multiple--valued map naturally defines an integer rectifiable current,
and explicit formulas for the boundary, the mass, and the first variations were
derived.

In \cite{DeLellis2011a}, De Lellis, Focardi, and Spadaro introduced a notion of
quasiconvexity for $Q$--integrands. This condition extends Morrey's classical
quasiconvexity to energies defined on Sobolev spaces of multiple--valued maps
and provides a necessary and sufficient condition for weak lower semicontinuity
of the associated functional.

A natural question is whether minimizers of such functionals exhibit regularity
properties. In \cite{Spadaro2015}, Spadaro asked whether partial regularity
results can be obtained for minimizers of quasiconvex functionals in spaces of
$Q$--valued maps, and whether there exists a distinguished subclass of
integrands for which such regularity holds.

The present work contributes to this program by identifying a structural
condition on the classical integrand—uniform polyconvexity—that is equivalent to
uniform quasiconvexity of the associated $Q$--integrands (see
Theorem~\ref{thm:main3}). By classical theory, uniform quasiconvexity ensures
weak lower semicontinuity of the corresponding energy functional. Thus, the
third result of the present work provides a foundation for studying existence,
and potentially regularity, of minimizers within the class of integrands
considered in the theorem below.

\begin{theorem*}[\protect{see~\ref{thm:main3} for the precise statement}]
    Let $\Psi : \ograss{n}{k}^{+} \to \R$ be a Lipschitz geometric integrand,
    and let $\psi$ be the associated classical integrand. For each $Q \in \natp$,
    let $\bar{\psi}_Q$ denote the corresponding $Q$-integrand. Then, for any
    constant $c>0$, $\psi$ is uniformly polyconvex if and only if $\bar{\psi}_Q$
    is uniformly quasiconvex for every $Q \in \natp$.
\end{theorem*}

Finally, we remark that in
\cite{degennaro2025nonpolyconvexqintegrandslowersemicontinuous} it was shown
that, for each fixed $Q \in \natp$, there exists a $Q$--integrand $\bar{\psi}_Q$
whose associated energy functional is lower semicontinuous in $W^{1,p}$,
although the corresponding classical integrand $\psi$ is not polyconvex. This
result shows that lower semicontinuity for a fixed $Q$ does not force
polyconvexity of the underlying classical integrand, highlighting the importance
of structural conditions that hold uniformly with respect to $Q$.

This paper is organised as follows. In Section~\ref{section:Preliminaries}, we
recall the basic terminology used throughout the
paper. Section~\ref{section:geom_integrands} introduces uniform polyconvexity
and uniform ellipticity with respect to polyhedral test pairs for geometric
integrands, and establishes the equivalence between these notions. In
Section~\ref{section:class_integrands}, we define uniform polyconvexity for
classical integrands. Sections~\ref{section:q_valued_functions} and
\ref{section:q_graphs} present the framework of multiple-valued functions and
the associated theory needed for the subsequent results. In
Section~\ref{section:ellipticity}, we define ellipticity with respect to
$Q$--graph test pairs, and Section~\ref{section:main_result} contains the proof
of the main result. Finally, Sections~\ref{section:q_energies} and
\ref{section:uniform_quasiconvexity} introduce $Q$--integrands and uniform
quasiconvexity, culminating in the proof of the third result.

\section{Preliminaries}\label{section:Preliminaries}

In~principle we follow the notation of Federer;
see~\cite[pp.~669--671]{Federer1969}. However, we use $\nat$ to denote the set
of positive integers. We use letters~$k$ and~$n$ to denote two integers
satisfying $0 \le k < n$. We say that a~function $f$ defined on some open subset
of a~Banach space with values in another Banach space is of class~$\cnt{l}$ if
it is $l$-times differentiable and $\uD^l f$ is~continuous;
cf~\cite[3.1.11]{Federer1969}.

\begin{definition}
    Let $\xi \in \tbwedge_k \R^n$. We define the vector space associated with
    $\xi$ by
    \begin{displaymath}
        \asssp \xi = \R^n \cap \bigl\{ v : v \wedge \xi = 0 \bigr\} \,.
    \end{displaymath}
\end{definition}

\begin{remark}[\protect{\cite[1.6.1]{Federer1969}}]
    An element $\xi \in \tbwedge_k \R^n$ is simple if and only if
    $\dim \asssp \xi = k$.
\end{remark}

\begin{definition}[\protect{\cite[3.2.28(b)]{Federer1969}}]
    We define the \emph{oriented Grassmannian $\ograss{n}{k}$} as
    \[
        \ograss{n}{k} = \tbwedge_k \R^n \cap \bigl\{ \xi : \xi \text{ is simple}
        ,\, |\xi| = 1 \bigr\} \,.
    \]
\end{definition}

\begin{remark}
    Let $\xi \in \tbwedge_k \R^n$. Recall from~\cite[1.8.1]{Federer1969} that
    $\|\xi\| = |\xi|$ if and only if $\xi$ is simple.
\end{remark}

\begin{definition}
    Given a vectorspace $X$ we write $\End{X}$ for the space of all linear maps
    mapping~$X$ to~$X$, i.e., $\End{X} = \Hom(X,X)$.
\end{definition}

\begin{definition}[\protect{\cite[1.7.4]{Federer1969}}]
    A linear map $j \in \Hom(\R^k,\R^n)$ is called an~\emph{orthogonal
    injection} if $j^* \circ j = \id{\R^k}$. If $p \in \Hom(\R^n,\R^k)$ and
    $j = p^*$ is an orthogonal injection, then we say that $p$ is
    an~\emph{orthogonal projection}. The set of all orthogonal projections
    $\R^n \to \R^k$ is denoted $\orthproj nk$.
\end{definition}

\begin{definition}[\protect{cf.~\cite[p.~14]{Rockafellar1970}}]
    Let $X$ be a vectorspace and $A \subseteq X$. We define
    \[
        \ray A = X \cap \{ ta : 0 \le t < \infty ,\, a \in A \}
        \quad \text{and} \quad
        \cone A = \conv (\ray A) \,.
    \]
\end{definition}

\begin{definition}
    Let $X$ be a normed vectorspace and $f : X \to \R$. We say that $f$ is
    \emph{positively homogeneous} if $f(tx) = t f(x)$ whenever $x \in X$ and
    $0 \le t < \infty$.
\end{definition}

\begin{definition}
    Let $X$ be a vectorspace and $f : X \to \R$. We say that $f$ is
    a~\emph{gauge} if it is convex non-negative and positively homogeneous.
    We~say that $f$ is a~\emph{strict gauge} if $f$ is a gauge and
    \[
        f(x+y) < f(x) + f(y)
        \quad \text{for all $x,y \in X \without \{0\}$ such that $x \notin \ray \{y\}$} \,.
    \]
\end{definition}

\begin{definition}[\protect{cf.~\cite[2.5.19]{Federer1969}}]
    Assume $X$ is a set and $x \in X$. The \emph{Dirac measure} over~$X$ having a~single
    atom at~$x$ is denoted by
    \[
        \Dirac{x}(A) = 
        \begin{cases}
          1 \text{\quad if \quad} x \in A \\
          0 \text{\quad if \quad} x \notin  A
        \end{cases}
        \quad \text{for $A \subseteq X$} \,.
    \]
\end{definition}

\begin{definition}
    A \emph{signed measure} over~$\ograss nk$ is a Daniell integral on the
    lattice of functions~$\ccspace{\ograss nk}$ as defined
    in~\cite[2.5.6]{Federer1969}.
\end{definition}

\begin{definition}[\protect{\cite[2.5.5]{Federer1969}}]
    \emph{Total variation} of a signed measure~$\mu$ over $\ograss nk$ is
    defined as
    \[
        \|\mu\|_{\mathrm{TV}} := |\mu|(\ograss{n}{k}), \quad \text{where} \quad
        |\mu| = \mu^{+}+\mu^{-} \,.
    \]
\end{definition}

\begin{definition}[\protect{cf. \cite[Ch.~1]{Rosa2023}}]
    Let $T$ be a $k$-rectifiable current represented by $(M,\tau,\Theta)$.
    The \emph{weighted Gaussian image} of $T$ is the measure $\gamma_T$ on
    $\ograss{n}{k}$ defined by
    \[
        \gamma_T := \tau_{\#}\bigl(\Theta\, \HM^k \restrict M\bigr).
    \]
\end{definition}

\section{Geometric Integrands}\label{section:geom_integrands}

\begin{definition}
    A \emph{geometric integrand} is positively homogeneous, non-negative and
    continious function $\Psi : \ray \ograss{n}{k} \to \R$.
\end{definition}

\begin{definition}
    A geometric integrand $\Psi$ is \emph{polyconvex} if it can be extended to a
    convex function on $(\tbwedge_k \R^n)$.
\end{definition}

\begin{definition}
    \label{def:upolyconvexity}
    Let $0 < c < \infty$ and let $\Psi$ be a geometric integrand.
    We say that $\Psi$ is \emph{uniformly polyconvex (with constant~$c$)} if
    \[
        \sum_{i=1}^{d} m_i \Psi(\eta_i) - \Psi(\eta_0)
        \;\ge\;
        c \bigl( \sum_{i=1}^{d} m_i |\eta_i| - |\eta_0| \bigr)
    \]
    whenever $d \in \nat$, $m_1,\ldots,m_d \in (0,\infty)$, and
    $\eta_0,\ldots,\eta_d \in \ograss{n}{k}$ satisfy
    $\eta_0 = \sum_{i=1}^{d} m_i \eta_i$.
    If $\Psi$ is uniformly polyconvex with constant $c$, we write
    $\Psi \in \UPC(c)$.
\end{definition}

\begin{remark}
    (\cite[3.22]{Lesniak_2025}). If $\Psi \in \UPC(c)$, then $\Psi$ is
    polyconvex.
\end{remark}

\begin{lemma}\label{lemma:upc_upc}
If $\Psi\in\UPC(c)$ for every $c<c_0$, then $\Psi\in\UPC(c_0)$.
\end{lemma}

\begin{proof}
    Assume toward a contradiction that $\Psi\notin\UPC(c_0)$.  
    Then, by definition, there exist
    \[
        d\in\nat,\qquad m_1,\ldots,m_d>0,\qquad
        \eta_0,\eta_1,\ldots,\eta_d\in\ograss nk
    \]
    with
    \[
        \eta_0=\sum_{i=1}^d m_i\,\eta_i
    \]
    such that the inequality for $\UPC(c_0)$ fails, i.e.
    \[
        \sum_{i=1}^d m_i\,\Psi(\eta_i)-\Psi(\eta_0)
        \;<\;
        c_0\Bigl(\sum_{i=1}^d m_i|\eta_i|-|\eta_0|\Bigr).
    \]
    Since $\Psi\in\UPC(c)$ for every $c<c_0$, we have
    \[
        \sum_{i=1}^d m_i\,\Psi(\eta_i)-\Psi(\eta_0)
        \;\ge\;
        c\Bigl(\sum_{i=1}^d m_i|\eta_i|-|\eta_0|\Bigr).
    \]
    Combining the two inequalities, we obtain, for every $c<c_0$
    \[
        c_0\Bigl(\sum m_i|\eta_i|-|\eta_0|\Bigr)
        \;>\;
        \sum m_i\Psi(\eta_i)-\Psi(\eta_0)
        \;\ge\;
        c\Bigl(\sum m_i|\eta_i|-|\eta_0|\Bigr).
    \]
    Letting $c\uparrow c_0$ the right-hand side converges to the left-hand 
    expression, contradicting the strict inequality. Hence 
    $\Psi\in\UPC(c_0)$.
\end{proof}

\begin{definition}[\protect{cf.~\cite[4.10]{Lesniak_2025}}]
    By $\PC{k}(U)$ we denote the space of polyhedral $k$-currents in 
    $U \subset \R^n$.
\end{definition}

\begin{definition}
    Recalling~\cite[4.1.8 and 4.1.32]{Federer1969} for $k \in \nat$ we set
    \[
        I_k = \underbrace{\Lbrack 0, 1 \Rbrack \times \cdots \times \Lbrack 0, 1 
        \Rbrack}_{\text{$k$ factors}} \in \IC_k(\R^k) \,.
    \]
\end{definition}

\begin{definition}[\protect{\cite[5.2]{Lesniak_2025}}]
    \label{def:polyhedral_test_pairs}
    We~say that $(S, D) \in \mathcal{P}_C$ is a~\emph{polyhedral test pair} if
    there exists $p \in \orthproj nk$ such that
    \begin{gather*}
        S \in \PC{k}(\R^n) \,,
        \quad
        D = (p^*)_{\#} I_k \,,
        \quad \text{and} \quad
        \partial S = \partial D  \,.
    \end{gather*}
\end{definition}

\begin{definition}
    Let $\Psi$ be a geometric integrand.
    For every $k$-rectifiable current $T$ represented by $(M,\tau,\Theta)$,
    we define the \emph{anisotropic energy functional} $E_{\Psi}$ by
    \[
        E_{\Psi}(T)
        := \int_M \Theta(x)\,\Psi\bigl(\tau(x)\bigr)\,\ud\HM^k(x).
    \]
\end{definition}

\begin{remark}
    By the definition of the push-forward measure (see \cite[2.1.2]{Federer1969}),
    the anisotropic energy functional can be expressed in terms of the weighted
    Gaussian image as
    \[
        E_{\Psi}(T)
        = \int_M \Theta\,\Psi\circ\tau \,\ud\HM^k
        = \int_{\ograss{n}{k}} \Psi \,\ud \tau_{\#}
        \bigl(\Theta\,\HM^k \restrict M\bigr)
        = \int_{\ograss{n}{k}} \Psi \,\ud \gamma_T .
    \]
\end{remark}

\begin{definition}\label{def:polyhedral_ellipticity}
    (\cite[5.3]{Lesniak_2025}).
    Let $\Psi$ be a geometric integrand and $\mathcal{P}_C$ be the set of 
    polyhedral test pairs. We say that $\Psi$ is \emph{elliptic with respect to 
    $\mathcal{P}_C$} if there exists $0 < c < \infty$ (the \emph{ellipticity 
    constant}) such that
    \[
        E_{\Psi}(S) - E_{\Psi}(D) \geq c \bigl( \Mass(S) - \Mass(D) \bigr)
        \quad \text{for $(S,D) \in \mathcal{P}_C$} \,.
    \]
    If this holds we write $F \in \AUE(\mathcal{P_C}, c)$.
\end{definition}

\begin{theorem}\label{thm:main1}
    Assume $\Psi$ is a Lipschitz geometric integrand, and let $c > 0$. Then
    \[
        \Psi \in \AUE(\mathcal{P_C}, c)
        \quad \Longleftrightarrow \quad
        \Psi \in \UPC(c).
    \]
\end{theorem}

\begin{proof}
    Suppose first that $\Psi \in \AUE(\mathcal{P_C}, c)$. 
    By Theorem \cite[5.10]{Lesniak_2025} and Lemma \ref{lemma:upc_upc}, we have
    $\Psi \in \UPC(c)$. Conversely, assume $\Psi \in \UPC(c)$. 
    Let $(S, D) \in \mathcal{P_C}$ be a polyhedral test pair. By the structure of 
    such test pairs and \cite[5.1.2]{Federer1969}
    \[
        \int \vec{S}(x)\ud\|S\|(x) = \int \vec{D}(x)\ud\|D\|(x) = \Mass(D)\eta_0,
        \quad \eta \in \ograss{n}{k},
    \]
    where $\eta_0 = \vec{D}$ for $\|D\|$--almost every $x\in\spt D$. Since 
    $\Psi \in \UPC(c)$, the function $G$ defined as
    \[
        G(\eta) := \Psi(\eta) - c|\eta|
    \]
    satisfies the Jensen--type inequality associated with uniform polyconvexity.
    Thus,
    \[
        \Mass(D) G(\eta_0) = G(\int \vec{S}(x)\ud\|S\|(x)) 
        \leq \int G(\vec{S}(x))\ud\|S\|(x),
    \]
    which is eqivalent to
    \[
        \Mass(D)\big( \Psi(\eta_0) - c|\eta_0| \big)
        \leq \int \Psi(\vec{S}(x))\ud\|S\|(x) - c\Mass(S).
    \]
    This shows that $\Psi \in \AUE(\mathcal{P_C}, c)$.
\end{proof}

\section{Classical Integrands}\label{section:class_integrands}

\begin{definition}
    Let $\eta_0 \in \Lambda_k(\R^n)$ be a fixed unit simple $k$-vector.
    The set of \emph{positively oriented $k$-vectors} with respect to $\eta_0$ 
    is defined by
    \[
        (\Lambda_k \R^n)^+_{\eta_0} :=
        \bigl\{
            \xi \in \Lambda_k(\R^n) : \xi \bullet \eta_0 > 0
        \bigr\}.
    \]
\end{definition}

\begin{definition}[\protect{cf. \cite[3.2.28(2)]{Federer1969}}]
    The \emph{positively oriented Grassmannian} with respect to $\eta_0$ is
    \[
        \ograss{n}{k}^{+}_{\eta_0} :=
        (\Lambda_k \R^n)^+_{\eta_0} 
        \cap 
        \bigl\{
            \xi \in \Lambda_k(\R^n) : \xi \text{ is simple and } |\xi| = 1
        \bigr\}.
    \]
\end{definition}

\begin{definition}
    Fix $\eta_0 \in \ograss nk$. We say that a geometric integrand $\Psi$ is
    \emph{uniformly polyconvex (with constant~$c$)} with respect to $\eta_0$ 
    and write $\Psi \in \UPC_{\eta_0}(c)$ if
    \[
        \sum_{i=1}^{d} m_i \Psi(\eta_i) - \Psi(\eta_0)
        \;\ge\;
        c \bigl( \sum_{i=1}^{d} m_i |\eta_i| - |\eta_0| \bigr)
    \]
    whenever $d \in \nat$, $m_1,\ldots,m_d \in (0,\infty)$, and
    $\eta_1,\ldots,\eta_d \in \ograss{n}{k}^{+}_{\eta_0}$ satisfy
    $\eta_0 = \sum_{i=1}^{d} m_i \eta_i$.
\end{definition}

\begin{definition}
    We define the class of geometric integrands that are uniformly polyconvex
    with constant $c > 0$ with respect to every reference plane
    $\xi \in \ograss{n}{k}$ by
    \[
        \UPC^{+}(c)
        :=
        \bigcap_{\xi \in \ograss{n}{k}} \UPC_{\xi}(c).
    \]
\end{definition}

\begin{remark}\label{lemma:plane_wise_upc}
    By the same argument as in the proof of Lemma~\ref{lemma:upc_upc}, the
    following properties hold
    \begin{itemize}
        \item[(i)] Fix $\eta_0 \in \ograss{n}{k}$.  
        If $\Psi \in \UPC_{\eta_0}(c)$ for every $c < c_0$, then
        $\Psi \in \UPC_{\eta_0}(c_0)$.
        \item[(ii)] If $\Psi \in \UPC^{+}(c)$ for every $c < c_0$, then
        $\Psi \in \UPC^{+}(c_0)$.
    \end{itemize}
\end{remark}

\begin{remark}
    Let $e_1,\ldots,e_n$ denote the standard basis of $\R^n$.  
    We set
    \[
        \xi_0 := e_1 \wedge e_2 \wedge \dots \wedge e_k,
    \]
    and define the shorthand notation
    \[
        \ograss{n}{k}^{+} := \ograss{n}{k}^{+}_{\xi_0}, \qquad
        (\Lambda_k \R^n)^+ := (\Lambda_k \R^n)^+_{\xi_0}.
    \]
\end{remark}

\begin{definition}
    \label{def:classical_integrand}
    A \emph{classical integrand} is a continuous, non-negative function
    \[
        \psi : \R^{(n-k)\times k} \to \R.
    \]
\end{definition}

\begin{definition}[\protect{\cite[1.6]{Rosa2023}}]
    For $X \in \R^{(n-k)\times k}$, let $\id{k}$ denote the $k \times k$
    identity matrix and define
    \[
        M(X) :=
        \begin{pmatrix}
            \id{k} \\
            X
        \end{pmatrix}
        \in \R^{n\times k}.
    \]
    Writing $w_1(X),\ldots,w_k(X)$ for the columns of $M(X)$, we define
    \[
        \tbwedge M(X) := w_1(X)\wedge \cdots \wedge w_k(X)
        \in \Lambda^k(\R^n).
    \]
    This yields a map
    \[
        \tbwedge M :
        \R^{(n-k)\times k}
        \longrightarrow
        \Lambda^k(\R^n).
    \]
\end{definition}

\begin{remark}[\protect{\cite[1.6]{Rosa2023}}]
    In the standard basis of $\Lambda^k(\R^n)$, the coordinates of
    $\bigwedge M(X)$ are precisely the minors of $X$ of all orders, which
    coincide with the $k\times k$ minors of the matrix $M(X)$.
\end{remark}

\begin{definition}
    A classical integrand $\psi$ is \emph{polyconvex} if there exists a gauge
    \begin{displaymath}
        \Psi : (\tbwedge_k \R^n)^+ \longrightarrow \R
    \end{displaymath}
    such that $\psi = \Psi \circ \tbwedge M$.
\end{definition}

\begin{definition}
    A classical integrand $\psi$ is \emph{uniformly polyconvex} if it is
    polyconvex and the associated geometric integrand $\Psi \in \UPC_{\xi_0}(c)$.
    If $\psi$ is uniformly polyconvex with constant $c$, we write
    $\psi \in \UPC^{\text{cl}}_{\xi_0}(c)$.
\end{definition}

\begin{remark}\label{rem:class_geom_poly_equivalence}
    By definition, if $\psi = \Psi \circ \tbwedge M$, then 
    $\psi \in \UPC^{\text{cl}}_{\xi_0}(c)$ if and only if 
    $\Psi \in \UPC_{\xi_0}(c)$.
\end{remark}

\section{Lipschitz $Q$-valued functions}\label{section:q_valued_functions}

\begin{remark}
    Throughout this paper we work in the framework of multiple-valued functions
    and the associated theory of currents developed in
    \cite{delellis2013multiplevaluedfunctionsintegral},
    \cite{De_Lellis_2011_Q_valued_functions_revisited},
    and \cite{DeLellis2011a}.
    Unless otherwise stated, measurability is always understood with respect to 
    the appropriate Hausdorff measure. Moreover, $\Sigma \subset \R^{n+m}$ 
    denotes a $\cnt{1}$, $m$-dimensional oriented submanifold.
\end{remark}

\begin{definition}(\cite[1.1]{DeLellis2011a})
    Let $Q \in \natp$. The metric space
    $(\qspace{Q}(\R^n), \mathcal{G})$ of unordered $Q$-tuples is defined by
    \[
        \qspace{Q}(\R^n)
        :=
        \left\{
        \sum_{i=1}^Q \llbracket P_i \rrbracket
        \;:\;
        P_i \in \R^n \text{ for } i = 1,\dots,Q
        \right\},
    \]
    where $\llbracket P_i \rrbracket$ denotes the Dirac mass at
    $P_i \in \R^n$.
    For any two elements
    \[
        T_1 = \sum_{i=1}^Q \llbracket P_i \rrbracket,
        \qquad
        T_2 = \sum_{i=1}^Q \llbracket S_i \rrbracket,
    \]
    the distance is defined by
    \[
        \mathcal{G}(T_1,T_2)
        :=
        \min_{\sigma \in \mathcal{P}_Q}
        \left(
        \sum_{i=1}^Q \lvert P_i - S_{\sigma(i)} \rvert^2
        \right)^{1/2},
    \]
    where $\mathcal{P}_Q$ denotes the group of permutations of $\{1,\dots,Q\}$.
\end{definition}

\begin{lemma}(\cite[1.1]{delellis2013multiplevaluedfunctionsintegral})\label{lem:1.1}
    Let $M \subset \Sigma$ be measurable and let
    $F : M \to \qspace{Q}(\R^n)$ be Lipschitz.
    Then there exist a countable partition $\{M_i\}_{i\in \nat}$ of $M$ into bounded
    measurable subsets and Lipschitz functions
    $f_i^j : M_i \to \R^n$ ($j = 1,\ldots,Q$) such that:
    \begin{enumerate}
        \item
        $F|_{M_i} = \sum_{j=1}^Q \llbracket f_i^j \rrbracket$ for every $i$, and
        $\operatorname{Lip}(f_i^j) \le \operatorname{Lip}(F)$ for all $i,j$;

        \item
        for every $i$ and $j,j' \in \{1,\ldots,Q\}$, either
        $f_i^j \equiv f_i^{j'}$ or
        $f_i^j(x) \neq f_i^{j'}(x)$ for all $x \in M_i$;

        \item
        for every $i$,
        \[
        DF(x)
        =
        \sum_{j=1}^Q \llbracket \uD f_i^j(x) \rrbracket
        \quad \text{for a.e.\ } x \in M_i .
        \]
    \end{enumerate}
    Moreover, for $\HM^m$-almost every $x$, the differential $DF(x)$ is well
    defined independently of the chosen decomposition.
\end{lemma}

\begin{definition}(\cite[1.2]{delellis2013multiplevaluedfunctionsintegral})
    Let $M \subset \Sigma$ be measurable and let
    $F \colon M \to \qspace{Q}(\R^n)$ be a Lipschitz map.
    We say that $F$ is \emph{proper} if the following holds:
    for every Lipschitz decomposition given by Lemma~\ref{lem:1.1}, namely,
    for every bounded measurable set $M_i \subset M$ and Lipschitz functions
    $f_i^j \colon M_i \to \R^n$ ($j=1,\dots,Q$) such that
    \[
        F|_{M_i} = \sum_{j=1}^Q \llbracket f_i^j \rrbracket,
    \]
    and for every compact set $K \subset \R^n$, the set
    \[
        \bigcup_{j=1}^Q \overline{(f_i^j)^{-1}(K)}
    \]
    is compact in $M_i$.
\end{definition}

\section{$Q$-graphs of Lipschitz functions}\label{section:q_graphs}

\begin{definition}(\cite[1.3]{delellis2013multiplevaluedfunctionsintegral})
    Let $M \subset \Sigma$ be measurable, and
    $F : M \to \mathcal A_Q(\R^n)$ a proper Lipschitz map.
    The push-forward of $M$ through $F$ is the $m$-current
    \[
        T_F := \sum_{i,j} (f_{ij})_\# \llbracket M_i \rrbracket,
    \]
    where $\{M_i,f_{ij}\}$ is any decomposition as in Lemma~\ref{lem:1.1}.
\end{definition}

\begin{remark}\label{rem:action_of_T}
    For a smooth compactly supported differential $m$-form 
    $\omega \in \mathcal{D}^m(\R^n)$, the action of $T_F$ can be 
    written as
    \[
        T_F(\omega)
        := 
        \sum_{i\in\nat} \sum_{j=1}^Q \int_{M_i}
        \langle \omega(f_i^j(x)),\, \uD f_i^j(x)_\sharp\, \vec e(x) \rangle
        \ud \HM^m(x).
    \]
\end{remark}

\begin{lemma}(\cite[1.4]{delellis2013multiplevaluedfunctionsintegral})
\label{prop:1.4}
The action of $T_F$ given in Remark~\ref{rem:action_of_T} is independent of the
chosen partition $\{M_i\}$ and decomposition $\{f_i^j\}$.
Hence $T_F$ is an integer rectifiable current
\[
    T_F = (\operatorname{Im}(F),\, \Theta,\, \vec \tau),
\]
where:
\begin{enumerate}
    \item
        $\operatorname{Im}(F)
        = \bigcup_{x\in M} \operatorname{spt}(F(x))
        = \bigcup_i \bigcup_{j=1}^Q f_i^j(M_i)$
        is an $m$-rectifiable set;

    \item
        $\vec \tau$ is a Borel unit $m$-vector orienting $\operatorname{Im}(F)$, and for
        $\HM^m$-a.e.\ $p \in \operatorname{Im}(F)$,
        if $f_i^j(x) = p$ and $\uD f_i^j(x)\,\sharp\,\vec e(x) \neq 0$, then
        \[
            \vec \tau(p)
            = \pm \frac{\uD f_i^j(x)_\sharp\, \vec e(x)}
            {|\uD f_i^j(x)_\sharp\, \vec e(x)|};
        \]

    \item
        \[\Theta(p)
            = \sum_{i,j,\,x: f_i^j(x)=p}
            \left\langle
            \vec \tau(p),\,
            \frac{\uD f_i^j(x)_\sharp\, \vec e(x)}
            {|\uD f_i^j(x)_\sharp\, \vec e(x)|}
            \right\rangle .
        \]
    \end{enumerate}
\end{lemma}

\begin{definition}(\cite[1.10]{delellis2013multiplevaluedfunctionsintegral})\label{def:Q-graphs}
    Let $M \subset \Sigma$ be measurable, and let $f \colon M \to \mathcal A_Q(\R^n)$ 
    be a proper Lipschitz map, with the following decomposition
    \[
        f = \sum_{i=1}^Q \llbracket f_i \rrbracket.
    \]
    Define the map $\tilde{F} \colon M \to \qspace{Q}(\R^{m+n})$ by
    \[
        \tilde{F}(x) := \sum_{i=1}^Q \llbracket (x, f_i(x)) \rrbracket.
    \]
    The current $T_{\tilde{F}}$ associated with $\tilde{F}$ is called the \emph{graph current} of 
    $f$ and will be denoted by $G_f$.
\end{definition}

\begin{remark}
    By the area formula for $Q$-graphs (\cite[1.11]{delellis2013multiplevaluedfunctionsintegral}), 
    if $f$ is a proper Lipschitz map, then the associated graph current $G_f$ has 
    finite mass, given by
    \[
        \Mass(G_f)
        =
        \int_M \sum_{i=1}^Q
        \sqrt{\det\bigl(I + (\uD f_i)^T \uD f_i\bigr)} \ud \HM^m,
    \]
    where $\{f_i\}$ are as in Lemma~\ref{lem:1.1}.
\end{remark}

\section{Ellipticity}\label{section:ellipticity}

\begin{definition}
    \label{def:graph_test_pair}
    Let $\eta \in \ograss nk$ and let $D \subset \asssp \eta \subset \R^n$ 
    be a unit cube contained in the $k$-dimensional plane $\asssp \eta$. Let
    \[
        f, h \colon D \to \qspace{Q}((\asssp \eta)^{\perp})
    \] 
    be proper Lipschitz $Q$-valued functions, and assume that
    \[
        h(x) = \sum_{j=1}^{J} Q_j \llbracket a_j + L_j x \rrbracket,
        \qquad a_j \in \R^{\,n-k},\;
        L_j \in \R^{(n-k) \times k},\quad
        \sum_{j=1}^J Q_j = Q.
    \]
    Denote by $G_f$ and $G_h$ the associated graph currents in $\R^n$. 
    We say that $(G_f, G_h)$ is a \emph{graph test pair}, and write
    $(G_f, G_h) \in \mathcal{P}_G$, if $f$ and $h$ coincide on the boundary of $D$:
    \[
        f(x) = h(x) \quad \text{for all } x \in \partial D.
    \]
    In particular, this implies that $\partial G_f = \partial G_h$.
\end{definition}

\begin{definition}
    Fix $\eta_0 \in \ograss{n}{k}$. By $\mathcal{P}_G(\eta_0)$ we denote the
    class of graph test pairs whose base cube lies in the plane $\asssp\eta_0$.
\end{definition}

\begin{definition}
    \label{def:ellipticity}
    Let $\Psi$ be a geometric integrand and let $\mathcal{P}_G$ be the class of
    graph test pairs.
    We say that $\Psi$ is \emph{elliptic with respect to $\mathcal{P}_G$} if there
    exists a constant $c>0$ (the \emph{ellipticity constant}) such that
    \[
        E_{\Psi}(G_f) - E_{\Psi}(G_h)
        \geq
        c\bigl(\Mass(G_f) - \Mass(G_h)\bigr)
        \quad \text{for all } (G_f,G_h)\in\mathcal{P}_G.
    \]
    If this holds, we write $\Psi \in \AUE(\mathcal{P}_G,c)$.
\end{definition}

\section{Main Result}\label{section:main_result}

\begin{definition}[\protect{cf.~\cite[\S 3]{Rosa2023}}]
    Let $e_1,\ldots,e_n$ be the standard basis of $\R^n$, and let 
    $P~\subseteq~\R^n$ be a linear subspace of dimension $k$. We say that $P$ 
    has \emph{rational slope} if there exists $f~\in~\Hom(\R^n,\R^{n-k})$ such 
    that 
    \[
        P = \ker f \quad \text{and} \quad f(e_i) \in \Z^{\,n-k} \text{ for all } i = 1,\ldots,n.
    \]
\end{definition}

\begin{definition}[\protect{\cite[5.6]{Lesniak_2025}}]
    Let $u_1,\ldots,u_{n-k}$ be the standard basis of~$\R^{n-k}$ and
    $\ast : \tbwedge_{n-k} \R^n \to \tbwedge_k \R^n$ be the Hodge star defined with
    respect to the standard orientation of $\R^n$; cf.~\cite[1.7.8]{Federer1969}.
    Set
    \begin{gather}
        \xi_{f} = \ast \bigl( f^*(u_1) \wedge \cdots \wedge f^*(u_{n-k}) \bigr)
        \in \tbwedge_k \R^n
        \quad \text{for $f \in \Hom(\R^n,\R^{n-k})$} 
        \\
        \text{and} \quad
        Q(n,k) = \bigl\{ t \xi_f :
        t \in \R ,\,
        f \in \Hom(\R^n,\R^{n-k}) ,\, f^*(u_i) \in \integers^{n}
        \text{ for } i \in \{ 1,2,\ldots,n-k \}
        \bigr\} \,.
    \end{gather}
\end{definition}

\begin{lemma}[\protect{\cite[5.8]{Lesniak_2025}}]\label{lem:SQ_dense}
    There holds
    \[
        S_{\Q} = \bigl\{ w_1 \wedge \cdots \wedge w_k : w_1,\ldots,w_k \in \Q^n \bigr\}
        \subseteq Q(n,k)
        \subseteq S_{\R} 
        := \bigl\{ \eta \in \tbwedge_k \R^n : \eta \text{ is simple} \bigr\},
    \]
    and $S_{\Q}$ is dense in $S_{\R}$.
\end{lemma}

\begin{lemma}\label{lemma:approximate}
    Suppose $d\in\nat$, $m_1,\dots,m_d>0$, and
    $\eta_0,\eta_1,\dots,\eta_d\in\ograss nk$ satisfy
    \[
        \sum_{i=1}^d m_i\,\eta_i=\eta_0,
    \]
    where $\eta_1,\dots,\eta_d$ are positively oriented with respect to $\eta_0$.
    Then for every $\varepsilon>0$ there exist an integer
    $N$ with $0\le N-d\le\binom{n}{k}$ and elements
    $\tilde\eta_0,\dots,\tilde\eta_N\in Q(n,k)\cap\ograss nk$ such that
    \[
        |\tilde\eta_0-\eta_0|<\frac{\varepsilon}{2},
        \qquad
        |\tilde\eta_i-\eta_i|<\frac{\varepsilon}{2dm_i}
        \quad\text{for }i=1,\dots,d,
    \]
    and
    \[
        \tilde\eta_0=\sum_{i=1}^d m_i\tilde\eta_i+\sum_{i=d+1}^N m_i\tilde\eta_i,
    \]
    where $m_i\ge0$ for $i=d+1,\dots,N$ and each $\tilde\eta_i$ is positively
    oriented with respect to $\tilde\eta_0$.
\end{lemma}

\begin{proof}
    Fix $\varepsilon>0$.
    For $\varepsilon$ sufficiently small, we claim that there exist elements
    $\tilde\eta_i\in Q(n,k)\cap\ograss nk$ for $i=0,\dots,d$ such that the stated
    estimates hold and, in addition,
    \[
        \tilde\eta_i\bullet\tilde\eta_0>0
        \quad\text{for }i=1,\dots,d,
        \qquad\text{and}\qquad
        \Bigl(\tilde\eta_0-\sum_{i=1}^d m_i\tilde\eta_i\Bigr)\bullet
        \frac{\tilde\eta_0}{|\tilde\eta_0|}>0.
    \]
    By Lemma~\ref{lem:SQ_dense}, we may choose
    $\tilde\eta_0\in Q(n,k)\cap\ograss nk$ such that
    \[
        |\tilde\eta_0-\eta_0|<\frac{\varepsilon}{2}
        \quad\text{and}\quad
        (\tilde\eta_0-\eta_0)\bullet\eta_0>0.
    \]
    Set $\zeta:=\tilde\eta_0-\eta_0$. Using Lemma~\ref{lem:SQ_dense} again, we 
    choose $\tilde\eta_1,\dots,\tilde\eta_d\in Q(n,k)\cap\ograss nk$ such that
    \[
        m_i\,|\tilde\eta_i-\eta_i|
        <
        \frac{1}{d}\,\zeta\bullet\tilde\eta_0,
        \qquad i=1,\dots,d.
    \]
    For $\varepsilon$ sufficiently small, this choice guarantees that
    $\tilde\eta_i\bullet\tilde\eta_0>0$ for $i=1,\dots,d$.
    Moreover,
    \begin{align*}
        \Bigl(\tilde\eta_0-\sum_{i=1}^d m_i\tilde\eta_i\Bigr)\bullet\tilde\eta_0
        &=
        (\eta_0+\zeta)\bullet\tilde\eta_0
        -\sum_{i=1}^d m_i(\tilde\eta_i-\eta_i)\bullet\tilde\eta_0
        -\eta_0\bullet\tilde\eta_0
        \\
        &=
        \zeta\bullet\tilde\eta_0
        -\sum_{i=1}^d m_i(\tilde\eta_i-\eta_i)\bullet\tilde\eta_0
        \\
        &\ge
        \zeta\bullet\tilde\eta_0
        -\sum_{i=1}^d m_i|\tilde\eta_i-\eta_i|
        >0,
    \end{align*}
    by the choice of the approximations.

    Define the residual
    \[
        r:=\tilde\eta_0-\sum_{i=1}^d m_i\tilde\eta_i.
    \]
    Let $\{e_1,\dots,e_M\}$ be an orthonormal basis of
    $\bigwedge_k\R^n$, where $M=\binom{n}{k}$, with
    \[
        e_1:=\frac{\tilde\eta_0}{|\tilde\eta_0|}.
    \]
    Then there exist $\beta>0$ and $w'\in\operatorname{span}\{e_1\}^\perp$ such that
    \[
        r=\beta e_1+w'.
    \]
    Choose $\rho>0$ such that $|w'|<\beta\rho$.
    Select vectors $w_1,\dots,w_{M-1}\in\operatorname{span}\{e_1\}^\perp$ satisfying
    \[
        B(0,2\rho)\subset\operatorname{conv}\{\pm w_1,\dots,\pm w_{M-1}\}.
    \]
    Define
    \[
        t_j:=|\tilde\eta_0|(e_1+w_j),\qquad j=1,\dots,M-1.
    \]
    By Lemma~\ref{lem:SQ_dense}, each $t_j$ admits a rational approximation $s_j$
    such that
    \[
        |s_j-t_j|<\varepsilon_1,
    \]
    where $\varepsilon_1>0$ is chosen sufficiently small.
    If $\varepsilon_1<|\tilde\eta_0|/2$, then
    \[
        s_j\bullet e_1>\frac{|\tilde\eta_0|}{2}>0,
    \]
    so each $s_j$ is positively oriented with respect to $\tilde\eta_0$.

    Define
    \[
        u_j:=\frac{s_j}{s_j\bullet e_1}-e_1.
    \]
    For $\varepsilon_1$ sufficiently small one has $|u_j-w_j|<\rho$, hence
    \[
        B(0,\rho)\subset\operatorname{conv}\{\pm u_1,\dots,\pm u_{M-1}\}.
    \]
    Since $|w'|<\beta\rho$, there exist coefficients $y_j\ge0$ with
    $\sum_j y_j=1$ such that
    \[
        \frac{w'}{\beta}=\sum_{j=1}^{2M-2} y_j u_j.
    \]
    Consequently,
    \[
        r=\sum_{j=1}^{2M-2}\Bigl(\beta\frac{y_j}{s_j\bullet e_1}\Bigr)s_j,
    \]
    which expresses $r$ as a positive linear combination of rational, positively 
    oriented $k$-vectors. By Carath\'eodory’s theorem \cite[Theorem~17.1.2]{Rockafellar1970},
    at most $M$ coefficients are nonzero. Setting $\tilde\eta_{d+j}:=s_j$ and 
    $m_{d+j}:=\beta y_j/(s_j\bullet e_1)$ concludes the proof.
\end{proof}

\begin{theorem}[\protect{cf.~\cite[3.3]{Rosa2023}}]
    \label{thm:antonio}
    Let $\eta_0 \in \ograss nk$ and let $\eta_1, \dots, \eta_d \in Q(n,k) \cap \ograss nk$ 
    be positively oriented with respect to $\eta_0$, with positive weights 
    $m_1, \dots, m_d \in (0,\infty)$ satisfying
    \[
        \sum_{i=1}^d m_i \, \eta_i = \eta_0.
    \]
    Let $E \in \ograss kk$ denote the standard orientation of $\R^k$, and let $p \in \orthproj nk$ satisfy $\tbwedge_k p^*(E) = \eta_0$. Define
    \[
        D := (p^*)_\# I_k \quad \text{and} \quad \mu := \sum_{i=1}^d m_i \, \Dirac{\eta_i}.
    \]
    Then, for every $\varepsilon > 0$, there exists a polyhedral $k$-current
    $B_{\varepsilon} \in \PC{k}(\R^n)$ with rational coefficients such that
    \[
        \partial B_{\varepsilon} = \partial D, \quad
        \|\gamma_{B_{\varepsilon}} - \mu\|_{\mathrm{TV}} < \varepsilon,
    \]
    and all tangent planes of $B_\varepsilon$ are positively oriented with respect 
    to $\eta_0$.
\end{theorem}

\begin{lemma}[\protect{cf.~\cite[3.4]{Rosa2023}}]\label{lemma:antonio}
    Let $[0,1]^k$ be the unit cube in $\R^k \subset \R^n$. Suppose 
    $B \in \PC{k}(\R^n)$ is a polyhedral $k$-current with rational 
    coefficients such that
    \[
        \partial B = \partial I_k,
    \]
    and $\HM^k$-almost every tangent plane to $B$ is positively oriented 
    with respect to $\R^k$. Then there exists an integer $Q \neq 0$ and 
    a $Q$-valued Lipschitz function
    \[
        f : [0,1]^k \longrightarrow \qspace{Q}(\R^{n-k})
    \]
    such that
    \[
        B = Q^{-1} G_f,
    \]
    where $G_f$ denotes the graph current associated to $f$.
\end{lemma}

\begin{theorem}
    \label{thm:main2}
    Assume $\Psi$ is a Lipschitz geometric integrand and 
    $0 < c_{2} < c_{1} < \infty$. If 
    $\Psi \in \AUE(\mathcal{P}_G, c_{1})$, then 
    $\Psi \in \UPC^{+}(c_{2})$.
\end{theorem}

\begin{proof}
    We argue by contradiction. Assume that 
    $\Psi \in \AUE(\mathcal{P}_G, c_{1})$ is Lipschitz with 
    $\Lip \Psi = L < \infty$ and that $\Psi \notin \UPC^{+}(c_{2})$.
    By Kirszbraun's theorem~\cite[2.10.43]{Federer1969}, $\Psi$ can be extended 
    to all of $\tbwedge_k \R^n$ preserving the Lipschitz constant; we 
    denote this extension again by $\Psi$.

    Since $\Psi \notin \UPC^{+}(c_{2})$, there exist $d \in \nat$,
    $m_1, \dots, m_d > 0$, and $\eta_0, \eta_1, \dots, \eta_d \in \ograss nk$, with
    $\eta_1, \dots, \eta_d$ positively oriented with respect to $\eta_0$, such that
    \[
        \sum_{i=1}^d m_i \eta_i = \eta_0,
        \qquad
        \sum_{i=1}^d m_i \Psi(\eta_i) - \Psi(\eta_0)
        < c_2 \left( \sum_{i=1}^d m_i |\eta_i| - |\eta_0| \right).
    \]
    Set
    \[
        M := \sup \im \Psi, \qquad
        \varepsilon :=
        \frac{\frac12 (c_1 - c_2)\bigl(\sum_{i=1}^d m_i - 1\bigr)}
        {\frac32 L + M + M \binom{n}{k}^{1/2} + c_1}.
    \]
    By Lemma~\ref{lemma:approximate}, there exist
    $\tilde{\eta}_0, \dots, \tilde{\eta}_{N} \in Q(n,k) \cap \ograss nk$ such that
    \[
        |\eta_i - \tilde{\eta}_i| < \frac{\varepsilon}{2 d m_i} \quad (i=1,\dots,d), \qquad
        |\eta_0 - \tilde{\eta}_0| < \frac{\varepsilon}{2},
    \]
    and
    \[
        \tilde{\eta}_0 = \sum_{i=1}^{N} m_i \tilde{\eta}_i,
    \]
    with all $\tilde{\eta}_i$ positively oriented with respect to $\tilde{\eta}_0$.
    Define the measures
    \[
        \mu := \sum_{i=1}^d m_i \Dirac{\eta_i},
        \qquad
        \tilde{\mu} := \sum_{i=1}^{N} m_i \Dirac{\tilde{\eta}_i}.
    \]
    By Theorem~\ref{thm:antonio}, there exist polyhedral $k$-currents
    $B_\varepsilon, D \in \PC{k}(\R^n)$ with rational coefficients such 
    that
    \[
        \vec{D} = \tilde{\eta}_0, \qquad
        \spt D \text{ is a unit cube in } \asssp \tilde{\eta}_0, \qquad
        \partial B_\varepsilon = \partial D,
    \]
    and
    \[
        \|\gamma_B - \tilde{\mu}\|_{\mathrm{TV}} < \varepsilon.
    \]
    Let $R \in SO(n)$ be a rotation such that $R_{\#} D = \llbracket [0,1]^k \rrbracket$.
    Then $R_{\#} B$ is a polyhedral $k$-current with rational coefficients satisfying
    \[
        \partial (R_{\#} B) = \partial \llbracket [0,1]^k \rrbracket,
    \]
    and all tangent planes of $R_{\#} B$ are positively oriented with respect to
    $\R^k \subset \R^n$.
    By Lemma~\ref{lemma:antonio}, there exists a $Q$-valued Lipschitz function
    \[
        f : [0,1]^k \to \qspace{Q}(\R^{n-k})
    \]
    such that
    \[
        G_f = Q\, R_{\#} B_\varepsilon.
    \]
    Define the trivial $Q$-valued function
    \[
        h : [0,1]^k \to \qspace{Q}(\R^{n-k}),
        \qquad h(x) := Q \llbracket 0 \rrbracket.
    \]
    Rotating back, define
    \[
        \tilde{f}(x) := f(Rx), \qquad
        \tilde{h}(x) := h(Rx),
        \quad x \in R^{-1}([0,1]^k) \subset \asssp \tilde{\eta}_0,
    \]
    so that the associated graph currents satisfy
    \[
        G_{\tilde{f}} = (R^{-1})_{\#} G_f = Q B_\varepsilon,
        \qquad
        G_{\tilde{h}} = (R^{-1})_{\#} G_h = Q D.
    \]
    Since $(G_{\tilde{f}}, G_{\tilde{h}}) \in \mathcal{P}_G$ and
    $\Psi \in \AUE(\mathcal{P}_G, c_1)$, we obtain
    \[
        E_{\Psi}(G_{\tilde{f}}) - E_{\Psi}(G_{\tilde{h}})
        \ge c_1 \bigl( \Mass(G_{\tilde{f}}) - \Mass(G_{\tilde{h}}) \bigr).
    \]
    On the other hand, repeating verbatim the estimates of
    \cite[5.10]{Lesniak_2025} yields
    \begin{multline}
        c_{2}\bigl( \sum_{i=1}^d m_i - 1 \bigr)
        + \varepsilon \bigl( \tfrac32 L + M + M \binom{n}{k}^{1/2} \bigr)
        > E_{\Psi}(G_{\tilde{f}}) - E_{\Psi}(G_{\tilde{h}})
        > c_{1}\bigl( \sum_{i=1}^d m_i - 1 \bigr) - c_{1}\varepsilon.
    \end{multline}
    This is impossible by the choice of $\varepsilon$. Hence, the assumption that
    $\Psi \notin \UPC^{+}(c_2)$ is false, and therefore
    $\Psi \in \UPC^{+}(c_2)$.
\end{proof}

\begin{theorem}\label{corollary:thm2}
    Assume $\Psi$ is a Lipschitz geometric integrand, and let $c > 0$. Then
    \[
        \Psi \in \AUE(\mathcal{P}_G, c)
        \quad \Longleftrightarrow \quad
        \Psi \in \UPC^{+}(c).
    \]
\end{theorem}

\begin{proof}
    Suppose first that $\Psi \in \AUE(\mathcal{P}_G, c)$.
    Combining Theorem \ref{thm:main2} with Remark \ref{lemma:plane_wise_upc}, 
    we have
    \[
        \Psi \in \UPC^{+}(c).
    \]
    Conversely, assume $\Psi \in \UPC^{+}(c)$, and let 
    $(G_f, G_h) \in \mathcal{P}_G$ be a graph test pair. Since $(G_f, G_h)$ is 
    a test pair, repeating the same argument as in \ref{thm:main1} shows that 
    $\Psi \in \AUE(\mathcal{P}, c)$.
\end{proof}

\begin{remark}\label{rem:plane_wise_equivalence}
    Fix $\eta_0 \in \ograss{n}{k}$. Using the same arguments as in 
    Lemma~\ref{lemma:plane_wise_upc} and Theorem~\ref{thm:main2}, one obtains 
    the analogue of Theorem~\ref{corollary:thm2}. That is, the 
    uniform ellipticity of $\Psi$ on graphs over a single plane $\eta_0$ is 
    equivalent to the plane-wise uniform polyconvexity i.e.
    \[
        \Psi \in \AUE(\mathcal{P}_G(\eta_0), c)
        \quad \Longleftrightarrow \quad
        \Psi \in \UPC_{\eta_0}(c).
    \]
\end{remark}

\section{$Q$-Energies}\label{section:q_energies}

\begin{definition}[\protect{\cite[0.1]{DeLellis2011a}}]
A measurable map
\[
    F \colon (\R^{n-k})^Q \times (\R^{(n-k)\times k})^Q \to \R
\]
is called a \emph{$Q$-integrand} if it is invariant under permutations, i.e.,
for every permutation $\sigma$ of $\{1,\ldots,Q\}$,
\[
    F(a_1,\ldots,a_Q, A_1,\ldots,A_Q)
    =
    F\bigl(a_{\sigma(1)},\ldots,a_{\sigma(Q)},
           A_{\sigma(1)},\ldots,A_{\sigma(Q)}\bigr).
\]
\end{definition}

\begin{remark}
Let
    $f \colon [0,1]^k \to \qspace{Q}(\R^{\,n-k})$ be a proper 
    Lipschitz $Q$-valued function. By the decomposition lemma \ref{lem:1.1}, 
    its differential $\uD f(x)$ is an element of $\qspace{Q}(\R^{(n-k)\times k})$.
    More precisely, if
    \[
        f(x) = \sum_{i=1}^Q \llbracket f_i(x) \rrbracket,
    \]
    then
    \[
        \uD f(x) = \sum_{i=1}^Q \llbracket \uD f_i(x) \rrbracket.
    \]
\end{remark}

\begin{definition}
    Let $F$ be a $Q$-integrand. For a $Q$-valued function
    $f \colon [0,1]^k \to \qspace{Q}(\R^{\,n-k})$ with decomposition
    $f = \sum_{i=1}^Q \llbracket f_i \rrbracket$,
    we define the associated \emph{anisotropic $Q$-energy functional} by
    \[
        E_F(f)
        :=
        \int_{[0,1]^k}
        F\bigl(f_1(x),\ldots,f_Q(x),
        \uD f_1(x),\ldots,\uD f_Q(x)\bigr)\,\ud x.
    \]
\end{definition}

\begin{remark}
    The symmetry of $F$ ensures that the definition of $E_F(f)$ is independent 
    of the chosen decomposition of $f$.
\end{remark}

\begin{remark}\label{rem:geom_Q_integrand}
    Every geometric integrand
    \[
        \Psi : \ograss{n}{k}^{+} \to \R
    \]
    induces a $Q$-integrand in a natural way. Let 
    $\psi := \Psi \circ \tbwedge M$, and define
    \[
        \bar{\psi}_Q(X_1,\ldots,X_Q)
        :=
        \sum_{i=1}^Q \psi(X_i)
        \quad
        \text{for}
        \quad
        (X_1,\ldots,X_Q) \in (\R^{(n-k)\times k})^Q.
    \]
    Then the associated anisotropic $Q$-energy depends only on the differential
    of $f$ and is given by
    \[
        E_{\bar{\psi}_Q}(f)
        =
        \int_{[0,1]^k}
            \bar{\psi}_Q\bigl(\uD f(x)\bigr) \ud x \,.
    \]
\end{remark}

\section{Uniform Quasiconvexity}\label{section:uniform_quasiconvexity}

\begin{definition}[\protect{cf.~\cite[0.1]{DeLellis2011a}}]
    \label{def:uniform_quasiconvexity}
    Let
    \[
        F : (\R^{n-k})^Q \times (\R^{(n-k)\times k})^Q \to \R
    \]
    be a locally bounded $Q$--integrand.
    We say that $F$ is \emph{uniformly quasiconvex with constant $c>0$} if the
    following holds: for every decomposition $Q = Q_1 + \cdots + Q_J$, every 
    choice of vectors $a_j \in \R^{n-k}$ with $a_i \neq a_j$ for 
    $i \neq j$, and every $L_j \in \R^{(n-k)\times k}$, define the
     affine $Q$--valued map
    \[
        h(x) := \sum_{j=1}^J Q_j \llbracket a_j + L_j x \rrbracket .
    \]
    For every collection of Lipschitz maps
    \[
        f_j : [0,1]^k \to \qspace{Q_j}(\R^{n-k}),
        \qquad
        f_j|_{\partial [0,1]^k} = Q_j \llbracket a_j + L_j x \rrbracket ,
    \]
    set
    \[
        f(x) := \sum_{j=1}^{J} f_j(x).
    \]
    Then
    \[
        E_F(f) - E_F(h)
        \;\ge\;
        c \bigl( \Mass(G_f) - \Mass(G_h) \bigr),
    \]
    where $G_f$ and $G_h$ denote the graphs of $f$ and $h$, respectively.
    In this case we write $F \in \UQC(c; Q)$.
\end{definition}

\begin{remark}\label{rem:aue_uqc}
    Let $\bar{\psi}_Q$ be as in Remark~\ref{rem:geom_Q_integrand}.  
    By the area formula for $Q$-valued graphs \cite[1.11]{delellis2013multiplevaluedfunctionsintegral},  
    \[
        E_{\bar{\psi}_Q}(f) = E_{\Psi}(G_f)
    \] 
    for every Lipschitz $Q$-valued function $f$.  
    Accordingly, the condition $\bar{\psi}_Q \in \UQC(c; Q)$ for all $Q \in \natp$  
    is equivalent to 
    \[
        \Psi \in \AUE(\mathcal{P}_G(\xi_0), c),
    \] 
    with $\xi_0 := e_1 \wedge e_2 \wedge \dots \wedge e_k$.
\end{remark}

\begin{theorem}\label{thm:main3}
    Let $\Psi : \ograss{n}{k}^{+} \to \R$ be a Lipschitz geometric 
    integrand, and let $\psi := \Psi \circ \tbwedge M$ be the associated 
    classical integrand. For each $Q \in \natp$, let $\bar{\psi}_Q$ denote 
    the corresponding $Q$-integrand as in Remark~\ref{rem:geom_Q_integrand}. 
    Then, for any constant $c > 0$
    \[
        \psi \in \UPC^{\text{cl}}_{\xi_0}(c) \quad \Longleftrightarrow \quad
        \bar{\psi}_Q \in \UQC(c; Q) \text{ for every } Q \in \natp.
    \]
\end{theorem}

\begin{proof}
    The equivalence follows from a chain of equivalences, each justified by the cited remarks:
    \begin{align*}
        \psi \in \UPC^{\mathrm{cl}}_{\xi_0}(c)
        &\;\overset{\text{\ref{rem:class_geom_poly_equivalence}}}{\Longleftrightarrow}\;
        \Psi \in \UPC_{\xi_0}(c) \\[2mm]
        &\;\overset{\text{\ref{rem:plane_wise_equivalence}}}{\Longleftrightarrow}\;
        \Psi \in \AUE(\mathcal{P}_G(\xi_0), c) \\[1mm]
        &\;\overset{\text{\ref{rem:aue_uqc}}}{\Longleftrightarrow}\;
        \bar{\psi}_Q \in \UQC(c; Q) \quad \text{for every } Q \in \natp.
    \end{align*}
\end{proof}

\subsection*{Acknowledgements}
This research was supported by the \href{https://ncn.gov.pl/}{National Science Centre Poland} 
(grant number 2022/46/E/ST1/00328). The author sincerely thanks Sławomir Kolasiński 
for his invaluable guidance and insightful discussions throughout this work.

{\small
\bibliographystyle{halpha}
\addcontentsline{toc}{section}{\numberline{}References}
\bibliography{multigraphsandellipticity.bib}{}
}

\bigskip

{
\small \noindent
Maciej Lesniak
\\
\texttt{maciej.lesniak0@gmail.com}
}

\end{document}